\documentclass[12pt,reqno]{amsart}
\usepackage{amsfonts}
\usepackage{amsthm}
\usepackage{amsmath}
\usepackage{amscd}
\numberwithin{equation}{section}
\usepackage{hyperref}
\hypersetup{nesting=true,debug=true,naturalnames=true}
\usepackage{graphicx,amssymb,upref}
\usepackage{enumerate}
\usepackage[all]{xy}
\usepackage{color}
\usepackage{mathrsfs}
\usepackage{caption}
\usepackage{paralist}
\usepackage{color}
\usepackage{float}
\usepackage{tikz,pgf}
\usetikzlibrary{calc}
\usetikzlibrary{intersections,decorations.markings}
\usepackage{tikz-cd}
\usepackage{lineno}

\theoremstyle{definition}
\newtheorem{theorem}{\bf Theorem}[section]
\newtheorem*{theorem1}{\bf Theorem A}
\newtheorem*{theorem2}{\bf Theorem B}

\newtheorem{lemma}[theorem]{\bf Lemma}

\theoremstyle{definition}

\newtheorem{proposition}[theorem]{\bf Proposition}
\newtheorem{question}{\bf Question}
\newcommand{\mm}[1]{\mathrm{#1}}
\newcommand{\mb}[1]{\mathbb{#1}}
\newcommand{\mc}[1]{\mathcal{#1}}
\makeatletter
\@namedef{subjclassname@2020}{\textup{2020} Mathematics Subject Classification}
\makeatother

\begin{document}

\title[{\fontsize{7}{7}\selectfont Manifold structures on highly connected Poincar\'e complexes}]{Manifold structures on highly connected Poincar\'e complexes}

\author[{\fontsize{7}{7}\selectfont Wen Shen}]{Wen Shen}

%\address{College of Mathematics and Physics, Wenzhou University, Wenzhou, P.R.China}
\email{shenwen121212@163.com} 

%Poincar\'{e} complex,Framed manifold,

\begin{abstract}
This paper constructs numerous examples of highly connected Poincar\'{e} complexes, each homotopy equivalent to a topological manifold yet not homotopy equivalent to any smooth manifold. Furthermore, we determine the homotopy type of any closed $2k$-connected framed $(4k+2)$-manifold with Kervaire invariant one for $k=7,15,31$. 	
\end{abstract}


\maketitle

\section{Introduction}\label{intro-Sec1}

From the classic theorem \cite{KirSie1969}, every compact topological manifold is homotopy equivalent to a finite CW complex. Conversely, we raise the following question:
\begin{question}\label{question1}
	Is a finite CW complex homotopy equivalent to a topological manifold or a smooth manifold?
\end{question} 

 On the other hand, orientable closed manifolds satisfy Poincar\'e duality. This property leads us to focus on finite CW complexes that    
 satisfy Poincar\'e duality. A finite CW complex $X$ is called a $k$-dimensional Poincar\'e complex if there exists a fundamental class
  $[X]\in H_k(X;\mb{Z})$ such that the cap product $[X] \cap : H^q(X;\mb{Z})\to  H_{k-q}(X;\mb{Z})$ is an isomorphism for all $q$.  
 
 Let $\pi:E\to X$ be an $S^{m-1}$-fibration $\xi$. We denote the mapping cone $X\cup_\pi CE$ by $T(\xi)$ (the Thom space of the fibration). By \cite[Lemma I.4.3]{BrowW}, there exists an element $U\in H^m(T(\xi);\mb{Z})$ such that the cup product with $U$ 
$$\cup U:H^i(X;\mb{Z})\to H^{m+i}(T(\xi);\mb{Z})$$
is an isomorphism for all $i$ (i.e. the Thom isomorphism holds). 

According to \cite{Spi1967}, for a simply connected $k$-dimensional Poincar\'e complex $X$, there exists an $S^{m-1}$-fibration $\xi$ over $X$ 
such that every element in $H_{m+k}(T(\xi);\mb{Z})$ is a spherical class where 
\begin{enumerate}
\item $m$ is much larger than $k$ ($m\gg k$). 
		\item $H_{m+k}(T(\xi);\mb{Z})\cong H_{k}(X;\mb{Z})\cong \mb{Z}$ (Thom isomorphism)
		\item For space $Y$, a class in $H_i(Y;\mb{Z})$ is spherical if it lies in the image of the Hurewicz homomorphism. 
\end{enumerate}
We refer to such an $S^{m-1}$-fibration over $X$ as a Spivak normal fibration of $X$. In \cite{Spi1967}, Spivak showed that the Spivak normal fibration of $X$ is unique up to stable fiber homotopy equivalence. 

In differential topology and surgery theory, researchers study the obstructions to lifting the structure of a Poincar\'e complex to the stronger structure exhibited by a manifold.
 
 Let $\mm{BO}$, $\mm{BTOP}$ and $\mm{B}G$ be the classifying spaces for stable vector bundles, stable topological microbundles, and stable spherical fibrations respectively \cite[pp.232]{Rudyak}. There exist natural forgetful maps 
$$\mm{BO}\stackrel{\alpha^{\mm{O}}_{\mm{TOP}}}{\longrightarrow}\mm{BTOP}\stackrel{\alpha^{\mm{TOP}}_{{G}}}{\longrightarrow}\mm{B}G$$
 
 Novikov and Browder presented the following theorem independently.
 
 \begin{theorem}\label{Smooth}\cite{BrowW} \cite{Novi}
 	Let $X$ be a simply connected Poincar$\mm{\acute{e}}$ complex of dimension $m \ge  5$ with Spivak normal fibration which is stably fiber homotopy equivalent to the sphere bundle of a vector bundle $\eta$. If 
 	\begin{itemize}
 		\item[(1)] $m$ is odd, or
 		\item[(2)] $m = 4k$ and $\mm{Index}X = (L_k(p_1 (\eta^{-1}),\cdots , p_k (\eta^{-1})))[X]$,
 	\end{itemize}
then there is a homotopy equivalence $f : M\to X$, $M$ a smooth $m$-manifold, such that $\nu = f^\ast(\eta)$ is a normal bundle of $M$. \end{theorem}
 
  The above theorem can be extended to the topological case by  
 \cite{KirSie1969}.
 \begin{theorem}\cite{Brow1971}\label{PLTOP} 
 Let $X$ be a $1$-connected Poincar$\mm{\acute{e}}$ complex of dimension $m \ge  5$. Then $X$ is homotopy equivalent to a topological manifold if and only if the classifying map $X\to \mm{B}G$ of its Spivak normal fibration has a homotopy lift to $\mm{BTOP}$. 
  \end{theorem}

In this paper, we focus on CW complexes of the form:
\begin{equation}
	X=(S_1^n\vee S_2^n)\cup_{\beta} e^{2n}\quad  \label{inXconstruction}
\end{equation}
By Hilton's work, we have
\begin{equation}
	\pi_{2n-1}(S^n_1\vee S_2^n)=\pi_{2n-1}(S^n_1)\oplus  \pi_{2n-1}(S_2^n)\oplus \mb{Z} \label{homotopywedgeSn}
\end{equation}
Let $\iota^i_n$ represent a generator of $\pi_n(S^n_i)$ for $i=1,2$. The generator of the direct summand $\mb{Z}$ is represented by the Whitehead product $[\iota^1_n,\iota^2_n]$.
Thus, the attaching map $\beta:S^{2n-1}\to S_1^n\vee S_2^n$ can be expressed as a triple $(\beta_1,\beta_2,\beta_3)\in \pi_{2n-1}(S^n_1\vee S^n_2)$ where $\beta_1,\beta_2,\beta_3$ 
 are its three components. %Specifically, if $\beta_1=0$, $\beta_2=0$, and $\beta_3=[\iota_n^1,\iota_n^2]$, then $X$ is homotopy equivalent to $S^n\times S^n$.

Now we state the main theorem:
\begin{theorem1}
Let $n$ be odd. Let $X$ denote the CW complex $$(S^{n}_1\vee S_2^{n})\cup_{([\iota_n^1,\iota_n^1],[\iota_n^2,\iota_n^2],[\iota_n^1,\iota_n^2])}e^{2n}$$
where $\iota_n^i\in \pi_{2n-1}(S^n_i)$ is a generator for $i=1,2$; $[-,-]$ denotes the Whitehead product. The following statements hold:
\begin{enumerate}
	\item $X$ is homotopy equivalent to a closed topological manifold. Moreover, for $n\ge 5$, the topological manifold is unique up to homeomorphism.
	\item If $n\ne 1,3,7,15,31,63$, $X$ is not homotopy equivalent to any smooth manifold.
	\item If $n=15,31,63$, $X$ is homotopy equivalent to a closed framed manifold with  Kervaire invariant one.
\end{enumerate}   
\end{theorem1}

For odd $n\ne 2^i-1$, the CW complex $X$ defined in Theorem A is known to be homotopy equivalent to a topological manifold by \cite{Shen}. This paper focuses on the case where $n$ takes the form $2^i-1$ for $i\ge 4$. 

For $n=1,3,7$, Adams \cite{Adams1960} showed that the Whitehead product $[\iota_n,\iota_n]$ is the zero element in $\pi_{2n-1}(S^n)$. By \cite[Corollary 3.5]{WhiGW1946},
 the CW complex $X$ as defined in Theorem A is homotopy equivalent to $S^n\times S^n$.
 
 Clearly, (2) and (3) of Theorem A cannot be directly derived from Theorem \ref{Smooth}. Instead, we will employ techniques from homotopy theory and foundational results in surgery theory to establish them.   

In 1960, Kervaire \cite{Kerva1960} introduced an invariant for almost framed $(4k + 2)$-manifolds. This invariant, known as the Kervaire invariant, can be defined in a general manner: it corresponds to the surgery obstruction in the $(4k+2)$-dimensional framed bordism group (see \cite{KervaMilnor, BrowW}).
 Based on previous research: Browder \cite{BrowderW1969}, Mahowald-Tangora \cite{MahT}, Barratt-Jones-Mahowald \cite{BaJM}, Hill-Hopkins-Ravenel \cite{HiHR}, and Lin-Wang-Xu \cite{LWX}, the dimensions for which there exist framed manifolds of Kervaire invariant one are $2, 6, 14, 30, 62$, and $126$.

In dimensions $2$, $6$, and $14$, the sphere products $S^1\times S^1$, $S^3\times S^3$, and $S^7\times S^7$ can be framed to have Kervaire invariant one.
In dimension $30$, Jones explicitly constructed a framed manifold of Kervaire invariant one in \cite{Jones}.
 From (3) of Theorem A, we obtain one of homotopy types of framed $2n$-manifolds with Kervaire invariant one for $n=15,31,63$.

Let $k\ge 1$.
By surgery theory \cite{KervaMilnor}, every bordism class in the framed bordism group $\Omega_{4k+2}^{\mm{fr}}$ can be represented by a closed, $2k$-connected smooth $(4k+2)$-manifold with a given framing. Next, we give the homotopy type of
such a manifold with Kervaire invariant one.  
%Therefore, (1) and (3) of Theorem A imply the following:

%a closed, $2k$-connected framed $(4k+2)$-manifold with Kervaire invariant zero is homeomorphic to the connected sum of $s$ copies of $S^n\times S^n$, denoted by $\#_{i=1}^{s}S_i^n\times S_i^n$, where $2s>0$ equals the $(2k+1)$-th Betti number of the framed manifold. Since 

%if and only they admit the same $(2k+1)$-th Betti number.

%By the minimal cell structure \cite[Proposition 4.1]{Wall1965}, the homotopy type of $M$ has the form:
%$$(\vee_{i=1}^{s}(S_{i}^{2k+1}\vee S_{i}^{2k+1}))\cup_{\beta} e^{4k+2}$$ 
%where $S_{i}^{2k+1}$ denotes the standard $(2k+1)$-sphere $S^{2k+1}$, $2s$ equals the $(2k+1)$-th Betti number of $M$.
\begin{theorem2}
	Let $n=15,31,63$. Every closed, $(n-1)$-connected framed $2n$-manifold $M$ with Kervaire invariant one is homeomorphic to $(\#_{i=1}^{s-1}S_i^n\times S_i^n)\#N$, where 
	\begin{itemize}
		\item $N$ is the topological $2n$-manifold homotopy equivalent to the CW complex defined in Theorem A;
		\item $2s$ equals the $n$-th Betti number of $M$. %$\mm{rank}(H_n(M;\mb{Z}))$.
	\end{itemize}
\end{theorem2}

The plan of this paper is as follows: In Section \ref{TKIn}, we define the topological Kervaire invariant on the CW complex $X$ as \eqref{inXconstruction}. In Section \ref{SNfi}, we analyze the Spivak normal fibration of $X$. Finally, we use the topological Kervaire invariant to characterize the obstruction for $X$ to be homotopy equivalent to a smooth manifold and prove Theorem A and Theorem B in Section \ref{Obs-section}. 

\section{Topological Kervaire invariant}\label{TKIn}

In this section, by following \cite{Kerva1960}, we primarily construct topological Kervaire invariants on certain CW complexes of the form \eqref{inXconstruction}. Before the construction, we first make some preparatory remarks. 

Let $n$ be an odd integer. Recall the EHP sequence \cite{Jame} 
$$\pi_{2n}(S^n)\stackrel{\Sigma}{\to} \pi_{2n+1}(S^{n+1})\stackrel{\mm{H}}{\to}\pi_{2n+1}(S^{2n+1})\to \pi_{2n-1}(S^{n})\stackrel{\Sigma}{\to } \pi_{n-1}^s\to 0$$
where $\mm{H}$ is defined by the Hopf invariant, $\Sigma$ is the suspension operation, $\pi_{n-1}^s$ denotes the $(n-1)$-dimensional stable homotopy group of the sphere. By Serre's work, $\pi_i(S^m)$ is finite except in two specific cases:  when $i = m$,
and when $m$ is even and $i=2m-1$ (see \cite{Toda}). Consequently, $\pi_{2n+1}(S^{n+1})$ has precisely one $\mb{Z}$-direct summand that can be detected by the Hopf invariant. It is well known that there exists $\alpha\in \pi_{2n+1}(S^{n+1})$ with Hopf invariant $2$ (see \cite{Hatcher}). 

By the well-known result in \cite{Adams1960}, for odd integers $n\ne 1,3,7$, there is no element of Hopf invariant one in $\pi_{2n+1}(S^{n+1})$. Using this fact together with the EHP sequence, we obtain the following lemma:
\begin{lemma}\label{pi2n-1Snexac}
	Let $n\ne 1,3,7$ be odd. The following sequence
	$$0\to \mb{Z}/2\to \pi_{2n-1}(S^{n})\stackrel{\Sigma}{\to } \pi_{n-1}^s\to 0$$
	is exact.
\end{lemma}

Let $\Omega=\Omega S^{n+1}$ denote the loop space of the $(n+1)$-sphere $S^{n+1}$. It is straightforward to verify that
$H^n(\Omega;\mb{Z})=\mb{Z}\langle e_n\rangle $, $H^{2n}(\Omega;\mb{Z})=\mb{Z}\langle e_{2n}\rangle$, and $H^{i}(\Omega;\mb{Z})=0$ for $n<i<2n$. 

By the minimal cell structure \cite[Proposition 4.1]{Wall1965}, the loop space $\Omega=\Omega S^{n+1}$ has the following weak homotopy type: 
\begin{equation}
	W(\Omega)= S^n\cup_\alpha e^{2n}\cup \cdots \label{cellOmega}
\end{equation}
where $\alpha$ is the attaching map, and the dimensions of the omitted cells are greater then $2n$. 
\begin{lemma}\label{alpha}
Let $n\ne 1,3,7$ be odd. 
\begin{enumerate}
	\item The attaching map $\alpha$ in \eqref{cellOmega} is homotopic to the Whitehead product $[\iota_n,\iota_n]$ where $\iota_n$ represents a generator of $\pi_n(S^n)$.
	\item $\Sigma \alpha=0\in \pi_{2n}(S^{n+1})$ where $\Sigma$ is the suspension operation.
\end{enumerate} 
\end{lemma}
\begin{proof}
Let $n\ne 1,3,7$ be odd.
 The weak homotopy type \eqref{cellOmega} implies 
$$\pi_{2n-1}(\Omega S^{n+1})\cong \pi_{2n-1}(S^n)/\langle \alpha \rangle$$
Note that $\pi_{2n-1}(\Omega S^{n+1})\cong \pi_{2n}(S^{n+1})\cong \pi_{n-1}^s$. By Lemma \ref{pi2n-1Snexac}, $\alpha$ is non-zero in $\pi_{2n-1}(S^n)$.

 By \cite[pp.335]{WhiteGW1978}, $\Sigma \Omega S^{n+1}$ has the weak homotopy type $\bigvee_{k=1}^\infty S^{nk+1}$. From the cell structure \eqref{cellOmega}, we have  $$\bigvee_{k=1}^\infty S^{nk+1}\simeq S^{n+1}\cup_{\Sigma \alpha} e^{2n+1}\cup \cdots$$
 This implies $\Sigma \alpha=0\in \pi_{2n}(S^{n+1})$. By the short exact sequence in Lemma \ref{pi2n-1Snexac}, $\alpha$ represents the generator of the subgroup $\mb{Z}_2\subset \pi_{2n-1}(S^n)$.

By \cite[Theorem 1.1.1]{Adams1960}, the Whitehead product $[\iota_n,\iota_n]$ is non-zero in $\pi_{2n-1}(S^n)$. Theorem 3.11 in \cite{WhiGW1946} induces $$\Sigma([\iota_n,\iota_n]) =0\in \pi_{2n}(S^{n+1})$$ 
 Therefore, by Lemma \ref{pi2n-1Snexac}, $\alpha=[\iota_n,\iota_n]\in \pi_{2n-1}(S^n)$.  
\end{proof}

Now we focus on the CW complex as follows:
\begin{equation}
	X=(S_1^n\vee S_2^n)\cup_{\beta} e^{2n}\quad \text{for odd $n\ne 1,3,7$} \label{Xconstruction}
\end{equation}
The attaching map $\beta$  represents the following homotopy class 
\begin{equation}
	(\beta_1,\beta_2,[\iota^1_n,\iota^2_n])\in \pi_{2n-1}(S^n_1\vee S^n_2) \label{compnentofbeta}
\end{equation} 
where $\beta_1=[\iota^1_n,\iota^1_n]$ or $0$, $\beta_2=[\iota^2_n,\iota^2_n]$ or $0$.

By \cite[Proposition 6.2]{Shen}, we have the following result:
\begin{proposition}\label{poincareX}
Let $X$ be a CW complex as  \eqref{Xconstruction}.
	For any generator $\ell_i\in H^n(S^n_i;\mb{Z})\subset H^n(X)$, the cup product $\ell_1\cup \ell_2$ is a generator of $H^{2n}(X;\mb{Z})=\mb{Z}$. In other words, $X$ is a Poincar\'e complex.
\end{proposition}

Next we define the topological Kervaire invariant for Poincar\'{e} complexes of the form \eqref{Xconstruction}.
\begin{lemma}\label{fmap}
	Let $X$ be a CW complex as  \eqref{Xconstruction}. For any generator $\ell_i\in H^n(S^n_i;\mb{Z})\subset H^n(X;\mb{Z})$, there exists a map $f:X\to \Omega$ such that its induced cohomology homomorphism satisfies $f^\ast(e_n)=\ell_i$. %where $e_n$ generates $H^n(\Omega;\mb{Z})$.
\end{lemma}
\begin{proof}
By collapsing the $S^n_{3-i}$ to a point, we obtain a map $$\mathscr{C}_i:S^n_1\vee S^n_2\to S^n_i$$ 
where $i=1,2$. Thus there exists a natural map 
$$f(i):S^n_1\vee S^n_2\stackrel{\mathscr{C}_i}{\longrightarrow }S^n_i\to W(\Omega)$$
whose induced cohomology homomorphism satisfies $f(i)^\ast(e_n)=\ell_i$ for $i=1,2$. Moreover, the homotopy homomorphism satisfies that $\iota_n=f(i)_\ast(\iota^i_n)$ is a generator of $\pi_n(W(\Omega))\cong \pi_n(S^n)$. 

 Recall Lemma \ref{alpha} and the components of $\beta$ (see \eqref{compnentofbeta}). Under the homomorphism
$$f(i)_\ast:\pi_{2n-1}(S^n_1\vee S^n_2) \to \pi_{2n-1}(W(\Omega))$$
we have $f(i)_\ast(\beta)=0$. Therefore, the map $f(i)$ can be extended to the CW complex $X$. For the weak homotopy equivalence $\Omega\to W(\Omega)$, the induced morphism $[X,\Omega]\to [X,W(\Omega)]$ is a bijection \cite[pp.182]{WhiteGW1978}. This completes the proof.
\end{proof}

We define a function $\varphi_0: H^n(X;\mb{Z}) \to  \mb{Z}/2$ as follows:

 Given a generator
 $\ell_i\in H^n(X;\mb{Z})$ for $i=1,2$, let $f: X\to \Omega $ be a map such that $f^\ast(e_n)=\ell_i$ (see Lemma \ref{fmap}). Then, we define 
 $$\varphi_0(\ell_i)
:= f^\ast(u_{2n}) [X]_2\in \mb{Z}/2$$ where $u_{2n}\in  H^{2n}(\Omega ; \mb{Z}/2)$ is the reduction modulo $2$ of $e_{2n} \in  H^{2n}(\Omega;\mb{Z})$, and $f^\ast(u_{2n}) [X]_2$ is the value of the cohomology class $f^\ast(u_{2n})$ on the generator $[X]_2$ of $H_{2n}(X;\mb{Z}_2)=\mb{Z}_2$.

\begin{lemma}
	The function $\varphi_0: H^n(X;\mb{Z}) \to  \mb{Z}/2$ is well-defined, i.e., $\varphi_0(\ell_i)$ is independent of the choice of the map $f:X\to \Omega$ satisfying $f^\ast(e_n)=\ell_i$.
\end{lemma} 
\begin{proof}
	Let $f,g:X\to \Omega$ be two maps such that $f^\ast(e_n)=g^\ast(e_n)\in H^n(X;\mb{Z})$. We need to show  $f^\ast(u_{2n})=g^\ast(u_{2n})\in H^{2n}(X;\mb{Z}/2)$. 
	
	By the cell structure of $X$ (see \eqref{Xconstruction}) and the condition $f^\ast(e_n)=g^\ast(e_n)$, there exists a homotopy $$F:X^{(2n-1)}\times I\to \Omega$$ such that $F|_{X^{(2n-1)}\times 0}=f|_{X^{(2n-1)}}$ and $F|_{X^{(2n-1)}\times 1}=g|_{X^{(2n-1)}}$ where $X^{(2n-1)}$ is the $(2n-1)$-skeleton of $X$.  Then we define a map 
	$$\hat{F}:X\times \mathring{I}\cup X^{(2n-1)}\times I\stackrel{f\cup g\cup F}{\longrightarrow}\Omega$$
	Let $Y=X\times I$, $A=X\times \mathring{I}\cup X^{(2n-1)}\times I$. The differential cochain $w^{2n}(f,g)$ of the pair $(f,g)$ with respect to $F$ is defined via the sequence:
	$$H_{2n+1}(Y,A;\mb{Z})\stackrel{h}{\leftarrow}\pi_{2n+1}(Y,A)\stackrel{\partial}{\to }\pi_{2n}(A)\stackrel{\hat F_\ast}{\to }\pi_{2n}(\Omega)$$
	where $h$ is the Hurewicz homomorphism. 
	
It is straightforward to verify that the following sequence
$$0\to H_{2n+1}(Y,A;\mb{Z})\cong \mb{Z}\to H_{2n}(A;\mb{Z})\cong \mb{Z}^2\to H_{2n}(Y;\mb{Z})\cong\mb{Z}\to 0$$
is split exact. Moreover, a generator $s_{2n+1}\in H_{2n+1}(Y,A;\mb{Z})$ is mapped to $(x_1,-x_2)\in H_{2n}(A;\mb{Z})$ where $x_1$ and $x_2$ are linearly independent generators of $H_{2n}(A;\mb{Z})$; both $x_1$ and $x_2$ map to the generator $x\in H_{2n}(Y;\mb{Z})\cong H_{2n}(X;\mb{Z})$.

Considering the commutative diagram
\[
\xymatrix@C=0.8cm{
\pi_{2n+1}(Y,A)\ar[d]^-{h}_-{\cong}\ar[r]^-{\partial}&\pi_{2n}(A) \ar[d]^-{h}\ar[r]^-{\hat F_\ast}&\pi_{2n}(\Omega)\ar[d]^-{h}\\
H_{2n+1}(Y,A;\mb{Z})\ar[r]^-{\partial_\#}& H_{2n}(A)\ar[r]^-{\hat F_\#}   & H_{2n}(\Omega;\mb{Z})=\mb{Z}
}
\]
we have the equation:
$$\hat F_\#\circ \partial_\#(s_{2n+1})=h[w^{2n}(f,g)(s_{2n+1})]$$
Since $\hat F_\#\circ \partial_\#(s_{2n+1})=f_\#(x)-g_\#(x)\in H_{2n}(\Omega;\mb{Z})$, we have
\begin{equation}
	f_\#(x)-g_\#(x)=h[w^{2n}(f,g)(s_{2n+1})] \label{differcoch}
\end{equation}
Let $\rho_2$ be the mod $2$ reduction.
By Serre's work, $\rho_2(h[w^{2n}(f,g)(s_{2n+1})])$ equals 
 the mod $2$ Hopf invariant of the element in $\pi_{2n+1}(S^{n+1})$ represented by $[w^{2n}(f,g)(s_{2n+1})]\in \pi_{2n}(\Omega)$. Since no element of odd Hopf invariant occurs in $\pi_{2n+1}(S^{n+1})$ (see \cite{Adams1960, Hatcher}), we have
 $$\rho_2(h[w^{2n}(f,g)(s_{2n+1})])=0\in H_{2n}(\Omega;\mb{Z}_2)$$
 Then Equation \eqref{differcoch} induces 
  $$\rho_2(f_\#(x)-g_\#(x))=0\in H_{2n}(\Omega;\mb{Z}_2)$$
Given $H_{2n}(X;\mb{Z}_2)\cong H_{2n}(\Omega;\mb{Z}_2)\cong \mb{Z}_2$, it follows that:
 $$f_\#=g_\#:H_{2n}(X;\mb{Z}_2)\to  H_{2n}(\Omega;\mb{Z}_2)$$
By taking the cohomology homomorphism, we finish the proof. 
\end{proof}

Following \cite[Proof of Lemma 1.3]{Kerva1960}, we have 
\begin{lemma}\label{varphiplus}
	For any $\ell_1,\ell_2\in H^n(X;\mb{Z})$, $$\varphi_0(\ell_1+\ell_2)=\varphi_0(\ell_1)+\varphi_0(\ell_2)+l_1\cdot l_2$$
	where $l_i$ is the mod $2$ reduction of $\ell_i$ for $i=1,2$; $l_1\cdot l_2$ is the value on the generator of $H_{2n}(X;\mb{Z}_2)$ of the cup product $l_1\cup l_2$.
\end{lemma}

By Lemma \ref{varphiplus}, $\varphi_0(2\ell)=0$ for any $\ell\in H^n(X;\mb{Z})$. Hence we can define a new function $\varphi: H^n(X;\mb{Z}/2) \to  \mb{Z}/2$ by $\varphi(l)=\varphi_0(\ell)$ where $l$ is the mod $2$ reduction of $\ell$. 

The function $\varphi: H^n(X;\mb{Z}/2) \to  \mb{Z}/2$ is then used to construct the number $\Phi^T(X)$ as follows. Take a basis $\{l_1,l_2\}$ of $H^n(X;\mb{Z}/2)$. Let $$\Phi^T(X):=\varphi(l_1)\cdot \varphi(l_2)$$
By Lemma \ref{varphiplus}, $\Phi^T(X)$ is independent of the choice of basis. We say $\Phi^T(X)$ the topological Kervaire invariant of $X$. 

\begin{proposition}\label{KervaofX}
	Let $X$ be the CW complex as    \eqref{Xconstruction} with the attaching map $\beta=(\beta_1,\beta_2,[\iota_n^1,\iota_n^2])$. 
	\begin{enumerate}
		\item If $\beta_1=[\iota_n^1,\iota_n^1]$ and $\beta_2=[\iota_n^2,\iota_n^2]$, $\Phi^T(X)=1$.
		\item  If $\beta_1=0$ and $\beta_2=[\iota_n^2,\iota_n^2]$, $\Phi^T(X)=0$. 
		\item If $\beta_1=[\iota_n^1,\iota_n^1]$ and $\beta_2=0$, $\Phi^T(X)=0$.
		\item If $\beta_1=0$ and $\beta_2=0$, $\Phi^T(X)=0$.
	\end{enumerate}
	
\end{proposition}
\begin{proof}
	Recall that $X=(S_1^n\vee S_2^n)\cup_{\beta} e^{2n}$ with the attaching map
	$$\beta=(\beta_1,\beta_2,[\iota^1_n,\iota^2_n])\in \pi_{2n-1}(S^n_1\vee S^n_2)$$ 
	  We define a map
	  $$\mathscr C:X{\longrightarrow}S^n_1\cup_{\beta_1}e^{2n}$$
	  by collapsing $S^n_2$ to a point, which  induces an isomorphism on $H^{2n}$. 
	  
When $\beta_1=[\iota_n^1,\iota_n^1]$,
Lemma \ref{alpha} guarantees the existence of a map
$$\mm{i}: S^n_1\cup_{\beta_1}e^{2n}\to \Omega$$ 
that induces isomorphisms on $H^n$ and $H^{2n}$. Moreover, $\ell_1=(\mm{i}\circ \mathscr C)^\ast(e_n)$ is a generator of $ H^n(X;\mb{Z})$.	  
 Since the composition $\mm{i}\circ \mathscr C$ induces an isomorphism on $H^{2n}$, we have 
$\varphi_0(\ell_1)=1$.

When $\beta_1=0$, we have $S^n_1\cup_{\beta_1}e^{2n}\simeq S^n_1\vee S^{2n}$.  Collapsing $S^{2n}$ to a point gives a map$$\mm{j}: S^n_1\cup_{\beta_1}e^{2n}\simeq S^n_1\vee S^{2n}\to \Omega$$
that induces an isomorphism on $H^n$. Thus, $\ell_1=(\mm{i}\circ \mathscr C)^\ast(e_n)\in H^n(X;\mb{Z})$ is a generator, and $(\mm{i}\circ \mathscr C)^\ast(e_{2n})=0\in H^{2n}(X;\mb{Z})$, which implies $\varphi_0(\ell_1)=0$. 

By symmetry, we obtain another generator $\ell_2\in H^n(X;\mb{Z})$ such that $\varphi_0(\ell_2)=1$ if $\beta_2=[\iota_n^2,\iota_n^2]$, and $\varphi_0(\ell_2)=0$ if $\beta_2=0$. 

Clearly, $\ell_1$ and $\ell_2$ form a basis of $H^n(X;\mb{Z})=\mb{Z}^2$. Since $\Phi^T(X)=\varphi_0(\ell_1)\cdot \varphi_0(\ell_2)$, the desired results follow.
\end{proof}

\section{Spivak normal fibration}\label{SNfi}

%The term ``fibration" always means a map $\pi:E\to X$ with the covering homotopy property for all spaces. All base spaces of fibrations are assumed to be path-connected. Thus the fiber $\pi^{-1}(x_1)$ is homotopy equivalent to $\pi^{-1}(x_2)$ for any $x_1,x_2\in X$. We say an $F$-fibration as a fibration whose fibers are homotopy equivalent to $F$.  

 In this section we 
 analyze the Spivak normal fibrations of Poincar\'{e} complexes of the form \eqref{Xconstruction}. 
 
 %First we introduce a key lemma:
% \begin{lemma}\cite{KaLaPoTe}\label{equivalence}
 	%Suppose $X$ is an $m$-dimensional Poincar\'e complex with a CW structure that has precisely one $m$-dimensional cell. Let $\phi: S^{m-1}\to X^{(m-1)}$ be the attaching map of this cell. Then $\phi$ is stably null homotopic if and only if the Spivak normal fibration of $X$ is trivial.
% \end{lemma}   

\begin{proposition}\label{Spivak11}
	Let $X$ be a CW complex as  \eqref{Xconstruction} with attaching map
	$\beta=([\iota_n^1,\iota_n^1],[\iota_n^2,\iota_n^2],[\iota_n^1,\iota_n^2])\in \pi_{2n-1}(S^n_1\vee S^n_2)$.
	 The Spivak normal fibration of $X$ is trivial.
\end{proposition}
\begin{proof} 
	It is evident that the Whitehead product is trivial upon application of the suspension operation. 
	So, $\beta$ is stably null homotopic. By \cite[Lemma 3.10]{KaLaPoTe}, the Spivak normal fibration of $X$ is trivial.
\end{proof}

\begin{proof}[Proof of Theorem A (1)]
 $X$ is a CW complex $$(S_1^n\vee S_2^n)\cup_{([\iota_n^1,\iota_n^1],[\iota_n^2,\iota_n^2],[\iota_n^1,\iota_n^2])} e^{2n}$$
 
	For odd $n=1,3,7$, Adams \cite{Adams1960} showed $[\iota_n,\iota_n]=0\in \pi_{2n-1}(S^n)$ where $\iota_n$ is a generator of $\pi_n(S^n)$. So $X$ is the homotopy type of $S^n\times S^n$.
	
For odd $n\ne 1,3,7$, by Proposition \ref{Spivak11} and Theorem \ref{PLTOP}, we completes the proof.

 Let $n\ge 5$ be odd, and let $M$ be a topological manifold homotopy equivalent to $X$. By \cite{KirSie1969}, $M$ admits a piecewise linear (PL) structure. By Sullivan's work \cite[Theorem 3]{Sull1966}, all PL manifolds homotopy equivalent to $X$ are PL homeomorphic, and also homeomorphic.	
\end{proof}

\section{Obstruction}\label{Obs-section}

Let $n\ne 1,3,7$ be odd.
By Proposition \ref{KervaofX} and (1) of Theorem A, the CW complex $(S_1^n\vee S_2^n)\cup_{([\iota_n^1,\iota_n^1],[\iota_n^2,\iota_n^2],[\iota_n^1,\iota_n^2])} e^{2n}$ is homotopy equivalent to a closed topological $(n-1)$-connected $2n$-manifold $\mc{M}$ with topological Kervaire invariant one. By \cite{KirSie1969}, $\mc{M}$ is a PL manifold. 
	
In this section, we address the question: is
  $\mc{M}$ a smooth manifold?

%\begin{lemma}\label{classnormal}
	%The classifying map for the stable normal bundle $\nu$ of $\mc{M}$ is homotopic to the following composition 
	%$$\mathfrak n:\mc{M}\stackrel{\mm{Col}}{\longrightarrow}S^{2n}\to \mm{BTOP}$$
	%where $\mm{Col}:\mc{M}\to S^{2n}$ collapses the $n$-skeleton $\mc{M}^{(n)}$ to a ponint.
%\end{lemma}
%\begin{proof}
	%By Lemma \ref{restrcinormal}, there exists a homotopy $F$ between the classifying map for the stable class represented by $\nu|_{\mc{M}^{(n)}}$ and the constant map. Then we can define the map $\bar H:\mc{M}^{(n)}\times I\cup \mc{M}\times \{0\}\to \mm{BTOP}$ by the homotopy $F$ and the classifying map for $\nu$. Note that $\bar H$ can be extended to $H:\mc{M}\times I\to \mm{BTOP}$. The map $H_1$ is the desired composition. This finishes the proof.  
%\end{proof} 

First, we establish a lemma that will be utilized in what follows.
%First, we prepare a lemma that will be used.
\begin{lemma}\label{Spivaknormal-manifold}
	  Let $M$ be a closed, simply connected smooth or PL manifold, and let $\nu$ be the normal bundle of an embedding of $M$ in Euclidean space. Then the spherical fibration obtained by deleting the zero-section from $\nu$
 is stably fiber homotopy equivalent to the Spivak normal fibration of the underlying Poincar\'e complex of $M$.
\end{lemma}
\begin{proof}
This lemma is referenced in \cite[pp.32]{Wall1970}.  
	%It proof follows by a standard application of the Thom-Pontrjagin construction (shrinking the complement of a tubular neighbourhood to a point).
\end{proof}

Next, we use the Kervaire invariant to characterize the obstruction for the topological manifold $\mc{M}$ to be a smooth manifold.

\begin{lemma}\label{prelemma}
Let $n\ge 3$ be odd, $M$ be a framed $(n-1)$-connected $2n$-manifold with $H_n(M;\mb{Z})=\mb{Z}^2$. Assume that the topological Kervaire invariant is defined on $M$, and $\Phi^T(M)=1$, then its Kervaire invariant $\Phi(M)$ equals $1$. 
\end{lemma}
\begin{proof}
	Suppose, for contradiction, that $\Phi(M)=0$. Then $M$ is cobordant to a homotopy $2n$-sphere $\Sigma^{2n}$ in the framed bordism group $\Omega_{2n}^{\mm{fr}}$. Let $-\Sigma^{2n}$ be the homotopy $2n$-sphere $\Sigma^{2n}$ with opposite orientation. Then $M\#(-\Sigma^{2n})$ is cobordant to the  zero class in $\Omega_{2n}^{\mm{fr}}$. 
	Note that $S^n\times S^n$ with trivial framing represents the zero class in $\Omega_{2n}^{\mm{fr}}$. So, $M\#(-\Sigma^{2n})$ is cobordant to $S^n\times S^n$ in $\Omega_{2n}^{\mm{fr}}$. 
	
	By \cite[Theorem 1]{Freed}, $M\#(-\Sigma^{2n})$ is diffeomorphic to $S^n\times S^n$. By Poincar\'e conjecture, $M\cong M\# S^{2n}$ is homeomorphic to $M\#(-\Sigma^{2n})$. This implies   
 $\Phi^T(S^n\times S^n)=\Phi^T(M)=1$.

	However, Proposition \ref{KervaofX} shows $\Phi^T(S^n\times S^n)=0$. This is a contradiction, and thus finishes the proof.  
\end{proof}

We now prove (2) and (3) of Theorem A.
\begin{proof}[Proof of Theorem A (2)]
	Let $n\ne 1,3,7,15,31,63$ be odd. By Theorem A (1), the CW complex $(S_1^n\vee S_2^n)\cup_{([\iota_n^1,\iota_n^1],[\iota_n^2,\iota_n^2],[\iota_n^1,\iota_n^2])} e^{2n}$ is the homotopy type of a closed topological $(n-1)$-connected $2n$-manifold $\mc{M}$ with topological Kervaire invariant one. 

Assume for contradiction that $\mc{M}$ is a smooth manifold, with stable normal bundle $\nu$.

Recall that the map $\alpha^\mm{O}_{G}:\mm{BO}\to \mm{B}G$ induces a homomorphism
$[-,\mm{BO}]\to [-,\mm{B}G]$ which maps a stable vector bundle to a stable spherical fibration by deleting the $0$-section. The map $\alpha^\mm{O}_{G}$ factors as $\alpha^\mm{O}_\mm{PL}:\mm{BO}\to \mm{BPL}$ and $\alpha^\mm{PL}_G:\mm{BPL}\to \mm{B}G$ where $\mm{BPL}$ is the classifying space of stable PL bundles.

Let $\mc{M}^{(k)}$ be the $k$-skeleton of the homotopy type of $\mc{M}$. We first show that the restriction bundle $\nu|_{\mc{M}^{(n)}}$ is trivial. 

Let $f:\mc{M}^{(n)}\to \mm{BO}$ classify $\nu|_{\mc{M}^{(n)}}$. 
%Under the forggetful maps $$\mm{BO}\stackrel{\alpha^\mm{O}_{\mm{TOP}}}{\longrightarrow}\mm{BTOP}\stackrel{\alpha^{\mm{TOP}}_G}{\longrightarrow} \mm{BG}$$
%the stable spherical fibration associated with $\nu$ is the Spivak normal fibration of the underlying Poincar\'e complex of $\mc{M}$. Therefore, 
By Lemma \ref{Spivaknormal-manifold} and Proposition \ref{Spivak11}, the composition
	$$\alpha^{\mm{PL}}_G\circ \alpha^\mm{O}_{\mm{PL}}\circ f:\mc{M}^{(n)}\to \mm{B}G$$
	is null homotopic.	
	Let $G/\mm{PL}$ be the fiber of  $\alpha^{\mm{PL}}_G$. Note that $\mc{M}^{(n)}=S^n\vee S^n$ where $n$ is odd. Since $\pi_n(G/\mm{PL})=0$ (see \cite{KirbSieb1977}), the composition $\alpha^\mm{O}_{\mm{PL}}\circ f$ is null homotopic. 

By \cite{Hir1963,HirMaz}, $\alpha^{\mm{O}}_{\mm{PL}}:\mm{BO}\to \mm{BPL}$ induces monomorphisms on $\pi_i$ for $i\ge 5$. Since $\alpha^\mm{O}_{\mm{PL}}\circ f$ is null homotopic, $f$ is also null homotopic. This implies that $\nu|_{\mc{M}^{(n)}}$ is trivial. 

Now we show that $\mc{M}$ is a framed manifold.

  A connected smooth manifold $M$ is {\it almost parallelizable} if the restriction bundle of the stable normal bundle of $M$ on $M \backslash x_0$
   for a point $x_0\in M$ is stably trivial. Since $\nu|_{\mc{M}^{(n)}}$ is trivial, $\mc{M}$ is almost parallelizable. By \cite[Lemma 8.2]{Shen}, $\mc{M}$ is a framed manifold.
 
Thus, by Lemma \ref{prelemma}, $\Phi(\mc{M})=1$. However, known results \cite{BrowderW1969, MahT, BaJM, HiHR, LWX} show framed manifolds of Kervaire invariant one exist only in dimensions 
$2, 6, 14, 30, 62$, and $126$. This contradicts with $\mm{dim}(\mc{M})=2n$ with $n\ne 1,3,7,15,31,63$.  
\end{proof}  

\begin{proof}[Proof of Theorem A (3)]
Let $n=15,31,63$. By Theorem A (1) and \cite{KirSie1969}, the CW complex $(S_1^n\vee S_2^n)\cup_{([\iota_n^1,\iota_n^1],[\iota_n^2,\iota_n^2],[\iota_n^1,\iota_n^2])} e^{2n}$ is the homotopy type of a closed PL $(n-1)$-connected $2n$-manifold $\mc{M}$ with topological Kervaire invariant one.  

Let $\mathfrak n:\mc{M}\to \mm{BPL}$ be the classifying map for the stable normal bundle of $\mc{M}$. By Lemma \ref{Spivaknormal-manifold} and Proposition \ref{Spivak11}, the composition 
$$\alpha_{G}^{\mm{PL}}\circ \mathfrak n:\mc{M}\to \mm{BPL}\to  \mm{B}G$$
is null homotopic.

 By \cite{BoVo}, $G\mm{/PL}$, $\mm{BPL}$ and $\mm{B}G$ are 
infinite loop spaces. Let $\mm{i}$ be the fiber inclusion $G\mm{/PL}\to \mm{BPL}$. By the same argument as in the proof of \cite[Lemma 3.1]{Shen}, there exists a map $\mathfrak f:\mc{M}\to G/\mm{PL}$ such that $\mm{i}\circ \mathfrak f$ is homotopic to $\mathfrak n$.
%$\mm{i}_\ast(\mathfrak f)=\mathfrak n\in  [\mc{M},\mm{BTOP}]$
 
Since $\mc{M}^{(n)}=S^n\vee S^n$ and $\pi_n(G/\mm{PL})=0$, the map $\mathfrak f$ is homotopic to the composition
$\mc{M}\stackrel{\mm{Col}}{\longrightarrow}S^{2n}\stackrel{f}{\longrightarrow} G/\mm{PL}$
where $\mm{Col}:\mc{M}\to S^{2n}$ collapses $\mc{M}^{(n)}$ to a point.
	Consequently, $\mathfrak n$ is homotopic to the  composition
\begin{equation}
		\mc{M}\stackrel{\mm{Col}}{\longrightarrow}S^{2n}\stackrel{f}{\longrightarrow} G/\mm{PL}\stackrel{\mm{i}}{\longrightarrow}\mm{BPL} \label{classifymapforM}
	\end{equation}
	
	Consider the following commutative diagram with exact rows
	\[
\xymatrix@C=2cm{
\pi_k(G/\mm{O})\ar[d]\ar[r]^-{}&\pi_k(G/\mm{PL})  \ar[d]^-{\mm{i}_\ast}\ar[r]^-{\partial}&\pi_{k-1}(\mm{PL/O}) \ar[d]^-{\cong}\\
\pi_{k}(\mm{BO}) \ar[r]^-{\alpha_{\mm{PL}\ast}^{\mm{O}}}&\pi_k(\mm{BPL}) \ar[r]^-{\mathfrak o}& \pi_{k-1}(\mm{PL/O})}
\]
where $\mm{PL/O}$ is the fiber of $\alpha_{\mm{PL}}^{\mm{O}}: \mm{BO}\to \mm{BPL}$. 

%By \cite{KirSie1969}, $\mm{PL/\mm{TOP}}\simeq \mm{K}(\mb{Z}/2,3)$. This induces $$\pi_k(G/\mm{TOP})\cong \pi_k(G/\mm{PL}),\quad \pi_k(\mm{TOP}/\mm{O})\cong \pi_k(\mm{PL}/\mm{O})$$ for $k\ge 6$. 

As noted in \cite{Bru1968}, $\mm{Im}(\partial)=bP_{k}\subset \Theta_{k-1}\cong \pi_{k-1}(\mm{PL/O})$ for $k\ge 6$, where $bP_k$
 denotes the group of homotopy spheres that bound parallelizable manifolds, and $\Theta_{k-1}$ is the group of diffeomorphism classes of homotopy $(k-1)$-spheres. By \cite{KervaMilnor}, for $k\equiv 2\pmod 4$ and $k\ge 6$, there exists an exact sequence
$$0\to bP_{k+1}\to \Theta_{k}\to \pi_k^s/J\stackrel{\Phi}{\to }\mb{Z}/2\to bP_k\to 0$$	
 where $\Phi$ is the Kervaire invariant. 
 
 By \cite{MahT,BaJM,LWX}, for $k=2n$, $\Phi$ is nontrivial, thus $bP_{2n}=0$. This implies that $\partial$ is trivial for $k=2n$.
 
 Since $2n\equiv 6\pmod 8$, $\pi_{2n}(\mm{BO})=0$. Thus $\mathfrak o$ is injective for $k=2n$. Hence, $\partial=\mathfrak o\circ \mm{i}_\ast$ implies that $\mm{i}_\ast$ is trivial for $k=2n$.

Therefore, the composition \eqref{classifymapforM} induces that the classifying map $\mathfrak n$ of the stable PL normal bundle of $\mc{M}$ is null homotopic. Consequently, $\mc{M}$ is a framed manifold with topological Kervaire invariant one. By Lemma \ref{prelemma}, $\Phi(\mc{M})=1$.	 
\end{proof}

Finally, we prove Theorem B.
\begin{proof}[Proof of Theorem B]
	Let $n=15,31,63$. By Theorem A (3), the CW complex $X$ defined in Theorem A is homotopy equivalent to a closed $(n-1)$-connected framed $2n$-manifold $M$ with $\Phi(M)=1$. Moreover, $\mm{rank}(H_n(M;\mb{Z}))=2$.
	
	Recall that the Kervaire invariant is additive with respect to the connected sum of framed manifolds.
	
	Let $N$ be another closed $(n-1)$-connected framed $2n$-manifold with $\Phi(N)=1$ and $\mm{rank}(H_n(N;\mb{Z}))=2s$ ($s>0$). Then we have
	$$\Phi((\#_{i=1}^{s-1}S_i^n\times S_i^n)\#M\#(-N))=0$$
	where $\#_{i=1}^{s-1}S_i^n\times S_i^n$ is equipped with the trivial framing.
	By the definition of the Kervaire invariant, $(\#_{i=1}^{s-1}S_i^n\times S_i^n)\#M\#(-N)$ is cobordant to a homotopy $2n$-sphere $\Sigma^{2n}$ in $\Omega_{2n}^{\mm{fr}}$, which implies  $(\#_{i=1}^{s-1}S_i^n\times S_i^n)\#M$ is cobordant to $N\# \Sigma^{2n}$ in $\Omega_{2n}^{\mm{fr}}$. 
	
	By \cite[Theorem 1]{Freed}, $(\#_{i=1}^{s-1}S_i^n\times S_i^n)\#M$ is diffeomorphic to $N\# \Sigma^{2n}$, and thus $(\#_{i=1}^{s-1}S_i^n\times S_i^n)\#M$ is homeomorphic to $N$. 
\end{proof}

\end{document}